%% file: Luke_Smith_Thesis_Arxiv.tex
\documentclass[12pt,fleqn]{Lukeucithesisarxiv}

\usepackage{amsmath}
\usepackage{amsthm}
\usepackage{amssymb, amsfonts}
\usepackage{array}
\usepackage{graphicx}
\usepackage{relsize}
\usepackage{url}

\usepackage{caption}
\usepackage{subcaption}  
\usepackage{multirow}
\usepackage{tabularx}
\usepackage{titlesec}
\usepackage{chngcntr}

\usepackage[plainpages=false,pdfborder={0 0 0}]{hyperref}

\newtheorem{theorem}{\textbf{Theorem}}[chapter]
\newtheorem{definition}{\textbf{Definition}}[chapter]

\newtheorem{lemma}[theorem]{Lemma}

\newcommand*{\defeq}{\mathrel{\vcenter{\baselineskip0.5ex \lineskiplimit0pt 
                     \hbox{\scriptsize.}\hbox{\scriptsize.}}}%
                     =} 

\counterwithout{footnote}{chapter}

\input{LukepreliminariesArxiv}

\hypersetup{
	pdftitle={\Thesistitle},
	pdfauthor={\Authorname},
	pdfsubject={\Degreefield},
}

\begin{document}

\preliminarypages

\include{IntroductionArxiv}
\include{Lukechapter1}
\include{Chapter2refining}
\include{Conclusion}


\clearpage
\phantomsection

\bibliographystyle{abbrv}
\bibliography{Lukethesis}



\end{document}

%% file: LukepreliminariesArxiv.tex
\thesistitle{Refining Multivariate Value Set Bounds}

\degreename{Doctor of Philosophy}

\degreefield{Mathematics}

\authorname{Luke Alexander Smith}

\committeechair{Professor Daqing Wan}
\othercommitteemembers
{
  Professor Alice Silverberg\\
  Professor Karl Rubin
}

\degreeyear{2015}

\copyrightdeclaration
{
  {\copyright} {\Degreeyear} \Authorname
}

 \prepublishedcopyrightdeclaration
 {
 	Portions of Introduction and Chapter 1 {\copyright} 2014 Elsevier Inc. \\
 	All other materials {\copyright} {\Degreeyear} \Authorname
 }

\dedications
{
  I would like to dedicate my work to my wife Stacie Ann Smith. I could not imagine my graduate career without her support, sense of pragmatism, organizational ability, and her love and friendship. 
	
	\vspace{1pc}
	
	I would also like to dedicate my work to my father, Jack Stuart Smith, who always believed in me unconditionally and fostered my education and imagination throughout my earlier years.
}

\acknowledgments
{
  I would like to express my appreciation and gratitude to the chair of my thesis committee, Professor Daqing Wan, who deepened my appreciation for number theory and provided me the opportunity to explore mathematical questions which sparked my interest. I also owe Daqing Wan for ideas which make up a majority Section \ref{newmethod}. \\
	
	I would also like to thank my committee members, Professor Alice Silverberg and Professor Karl Rubin, for their interest in my research and their advice throughout my coursework and my graduate career. \\
  
	In addition, I thank Elsevier Inc. for permission to include the portions of the Introduction and Chapter 1 which were originally published in their journal \textit{Finite Fields and Their Applications}, Vol. 28. \\
	
	Personal thanks goes out to Dennis Eichorn, Matt Keti, Roger Dellaca and Josh Hill for their guidance throughout the dissertation writing process and for being a part of my brainstorming process. \\
	
	Financial Support was provided by the University of California, Irvine Mathematics department and Professor Daqing Wan. \\

}

\thesisabstract
{Over finite fields, if the image of a polynomial map is not the entire field, then its cardinality can be bounded above by a significantly smaller value. Earlier results bound the cardinality of the value set using the degree of the polynomial, but more recent results make use of the powers of all monomials. 

\vspace{1pc}

In this paper, we explore the geometric properties of the Newton polytope and show how they allow for tighter upper bounds on the cardinality of the multivariate value set. We then explore a method which allows for even stronger upper bounds, regardless of whether one uses the multivariate degree or the Newton polytope to bound the value set. Effectively, this provides an alternate proof of Kosters' degree bound, an improved Newton polytope-based bound, and an improvement of a degree matrix-based result given by Zan and Cao.
}


%% file: IntroductionArxiv.tex
\chapter*{Introduction}
\addcontentsline{toc}{chapter}{Introduction}

\section{History and Motivation}

For a given polynomial $f(x)$ over a finite field $\mathbb{F}_q$, let $V_f \defeq $ Im$(f) $ denote the value set of $f$. Determining the cardinality and structure of the value set is a problem with a rich history and wide variety of uses in number theory, algebraic geometry, coding theory and cryptography.

\vspace{1pc}

Relevant to this paper are theorems which provide upper bounds on the cardinality of our value set when $f(x)$ is not a permutation polynomial.\footnote{Permutation polynomials have also been studied extensively in literature, in view of their application to cryptography and combinatorics. For more information about other ways value sets have been studied historically, please refer to \cite{Hill}.} Let $f(x) \in \mathbb{F}_q[x]$ be a single variable polynomial of degree $d > 0$ with $|V_f| < q$. Using the Chebotarev density theorem over rational function fields, S. D. Cohen proved in \cite{Cone} that there is a finite set of rational numbers $T_d \subset [0,1]$ (depending on degree $d$) such that \begin{eqnarray} \label{Cohen} |V_f| = c_fq + O_d(\sqrt{q}) \end{eqnarray} for some $c_f \in T_d$ depending on Gal($f(x) - t$)/$\mathbb{F}_q(t)$ and Gal($f(x) - t$)/$\overline{\mathbb{F}}_q(t)$. Guralnick and Wan refine this in \cite{Guralnick}, proving that for gcd($d,q$) = 1 and $|V_f| < q$, $|V_f| \leq \frac{47}{63}q + O_d(\sqrt{q})$. In addition, Mullen conjectured the bound \begin{eqnarray} \label{SingleWan} |V_f| \leq q - \frac{q-1}{d} \end{eqnarray} for non-permutation polynomials. This was proven by Wan, Shiue and Chen in \cite{WSCsharp} using $p$-adic liftings, but Turnwald later averted the use of liftings with a clever proof in \cite{Turnwald} using elementary symmetric polynomials. This bound was also proven sharp for any finite field by Cusick and M$\ddot{\textrm{u}}$ller (for $f(x) = (x+1)x^{q-1} \in \mathbb{F}_{q^k}[x]$, $|V_f| = q^k - \frac{q^k-1}{q}$ for all integers $k$, see \cite{sharp}). For more sharp examples, see \cite{WSCsharp}.

\vspace{1pc}

Despite the interest mathematicians have taken in the value set problem, most of the work in this area has been dedicated towards univariate polynomials. However, In the past 25 or so years, the multivariate value set problem has been addressed in a few different forms. It was first addressed by Serre in 1988 \cite{SerreGal} over varieties, in connection with Hilbert's irreducibility theorem and the inverse Galois problem. His theorem, alongside results by Fried \cite{FriedHilb} and by Guralnick and Wan \cite{Guralnick} give us upper bounds on our value set which generalize Cohen's result in (\ref{Cohen}). Though these results bound $|V_f|$ by some fraction of $|\mathbb{F}_q^n|$, it is important to note that the error terms in both results, though well behaved with respect to $q$, are exponentially large in terms of the degree $d$ of the map.

\section{Recent Multivariate Value Set Theorems} \label{recent}

A recently published paper by Mullen, Wan, and Wang in 2012 \cite{MWWValue} gives another bound on the value set of polynomial maps, one with no error terms: \begin{theorem} \label{MWWBound} Let $f(x_1,...,x_n)=(f_1(x_1,...,x_n),...,f_n(x_1,...,x_n))$ be a polynomial map over the vector space $\mathbb{F}_q^n$, and let \textnormal{deg} $f$ = \textnormal{max}$_i$ \textnormal{deg} $f_i$. \begin{eqnarray*} \textrm{If }|V_f| < q^n, \textrm{ then }|V_f| \leq q^n - \textnormal{min} \left\{ q, \hspace{.25pc} \frac{n(q-1)}{\textnormal{deg }f}\right\}. \end{eqnarray*} \end{theorem}Since the time their paper was published, multiple refinements have been made to this theorem.

\vspace{1pc}

One approach towards improving Theorem \ref{MWWBound} is to replace the term $\frac{n(q-1)}{\textnormal{deg }f}$ by using different properties of the polynomial map $f$. Note that the degree only takes one monomial of $f$ into account, so it is reasonable to expect tighter bounds on $|V_f|$ if we account for every monomial. In my first paper \cite{Mypaper}, I improved upon theorem \ref{MWWBound} by generalizing Mullen, Wan, and Wang's $p$-adic lifting approach and utilizing the Newton polytope $\Delta(f)$ of the  polynomial map $f$. The Newton polytope is constructed using all monomials of $f$ using discrete geometry, meaning it encodes more information than deg $f$ and allows for a stronger statement to be made: 

\begin{theorem}[Smith \cite{Mypaper}, 2014] \label{mygaptheorem} Let $f(x_1,...,x_n)=(f_1(x_1,...,x_n),...,f_n(x_1,...,x_n))$ be a polynomial map over the vector space $\mathbb{F}_q^n$, let $\Delta(f)$ be the Newton polytope of $f$, and let $\mu_f$ be a certain constant (defined explicitly later) dependent on $\Delta(f)$. \begin{eqnarray*} \textrm{If } |V_f| < q^n, \textrm{ then }|V_f| \leq q^n - \textnormal{min}\{q,\hspace{.25pc} \mu_f\cdot(q-1)\},\end{eqnarray*} \end{theorem} 

Zan and Cao also refine Thoerem 0.1 by using the degree matrix $D_f$ of the polynomial map $f$ in order to account for all of the monomials of $f$. Their approach generalizes the $p$-adic lifting technique as well and improves upon my statement in \cite{Mypaper}: \begin{theorem}[Zan, Cao \cite{ZanCao}, 2014] \label{omegagap} Let $f(x_1,...,x_n)=(f_1(x_1,...,x_n),...,f_n(x_1,...,x_n))$ be a polynomial map over the vector space $\mathbb{F}_q^n$ and let $D_f$ be the degree matrix of $f$. \begin{eqnarray*} \textrm{If } |V_f| < q^n, \textrm{ then }|V_f| \leq q^n - \textnormal{min}\{q,\hspace{.25pc} \omega_f\},\end{eqnarray*} where the constant $\omega_f$ (defined explicitly later) depends on $D_f$. \end{theorem} 

\vspace{1pc}

Overall, each new refinement gives us stronger bounds, i.e. $\omega_f \geq \mu_f \cdot (q-1) \geq \frac{n}{\textnormal{deg} f}(q-1)$ (see \cite{Adolph} and \cite{ZanCao}). In addition, in the univariate case, it has been shown that there are instances when $\omega_f$ is strictly larger than $\frac{q-1}{\textnormal{deg} f}$ (as opposed to $\mu_f$ always being equal to $\frac{1}{\textnormal{deg} f}$ when $n = 1$). However, since each of these bounds are of the form $|V_f| \leq q^n - \textnormal{min}\{C_f,q\}$ with $C_f$ dependent on the theorem, we are limited to removing at most $q$ elements from these cardinality bounds.

\vspace{1pc}

Another type of improvement on theorem \ref{MWWBound} removes this dependence on subtracting the minimum of two constants. Though still dependent on the polynomial map degree, a theorem by Kosters allows for a stronger bound whenever $n >$ deg $f$: \begin{theorem}[Kosters \cite{Kosters}, 2014]   Let $f(x_1,...,x_n)=(f_1(x_1,...,x_n),...,f_n(x_1,...,x_n))$ be a polynomial map over the vector space $\mathbb{F}_q^n$, and let \textnormal{deg} $f$ = \textnormal{max}$_i$ \textnormal{deg} $f_i$. \begin{eqnarray*} \textrm{If }|V_f| < q^n, \textrm{ then }|V_f| \leq q^n - \frac{n(q-1)}{\textnormal{deg } f}. \end{eqnarray*} \end{theorem} In order to achieve this result, Kosters completely averted the use of $p$-adic liftings, instead using a method more akin to Turnwald's univariate proof in \cite{Turnwald}.

\section{Main Result}

In this paper, I will refine these multivariate value set bounds even further, removing the minimum condition from theorems \ref{mygaptheorem} and \ref{omegagap}, ultimately proving the following theorem: \begin{theorem} Let $f(x_1,...,x_n)=(f_1(x_1,...,x_n),...,f_n(x_1,...,x_n))$ be a polynomial map over the vector space $\mathbb{F}_q^n$ and let $D_f$ be the degree matrix of $f$. \begin{eqnarray*} \textrm{If } |V_f| < q^n, \textrm{ then }|V_f| \leq q^n - \omega_f,\end{eqnarray*} where the constant $\omega_f$ depends on $D_f$. \end{theorem} To properly convey the significance of this bound in relation to prior bounds, we will describe the Newton polytope in Section \ref{NewPoly} and the degree matrix in Section \ref{Degree}. We will also define the constants associated with these objects and connections between the two. This manuscript will also contain portions of my work in \cite{Mypaper} which are relevant to the proofs of theorems in Chapter \ref{refine} as well as a proof of the main theorem in \cite{Mypaper} to highlight the difference in techniques used.

\section{The Newton Polytope}
\label{NewPoly}

Let $F$ be an arbitrary field and let $h \in F[x_1,...,x_n]$. If we write $h$ in the form \begin{eqnarray} \label{sparse} h(x_1,...,x_n) = \sum_{j=1}^m a_jX^{D_j}, \hspace{1 pc} a_j \in F^{*} \end{eqnarray} where \begin{eqnarray} \label{polynota} D_j = (d_{1j},...,d_{nj}) \in \mathbb{Z}_{\geq 0}^n, \hspace{1 pc} X^{D_j} = x_1^{d_{1j}}...x_n^{d_{nj}}, \end{eqnarray} then we have the following definition:

\begin{definition}[Newton polytope] The Newton polytope of polynomial $h \in F[x_1,...,x_n]$, $\Delta(h)$, is the convex closure of the set $\{D_1,...,D_m\} \cup \{(0,...,0)\}$ in $\mathbb{R}^n$.
\end{definition}

Geometric properties of the Newton polytope, such as its dilation by $k \in \mathbb{R}$, its volume or its decomposition into other polytopes via Minkowski Sum, are useful tools in discerning properties of their associated polynomials. For more information, see \cite{WeiCao}, \cite{zeta}, and \cite{variation}.

\vspace{1pc}

The significance of the Newton polytope to the multivariate value set problem comes from the definition of the following quantity:

\begin{definition}[Minimal dilation factor $\mu_h$] Let $F$ be a field, let $h \in F[x_1,...,x_n]$, and let $\Delta(h)$ be the Newton polytope of $h$. \begin{eqnarray*} \mu_h \defeq \textnormal{inf}\{ k \in \mathbb{R}_{>0} \mid k\Delta(h) \cap \mathbb{N}^n \neq \varnothing \}. \end{eqnarray*}
\end{definition}

In other words, $\mu_h$ is the infimum of all positive real numbers $k$ such that the dilation of $\Delta(h)$ by $k$ contains a lattice point with strictly positive coordinates, and we define $\mu_h = \infty$ if such a dilation does not exist. For our purposes, since the vertices of our polytopes have integer coordinates, $\mu_h$ will always be finite and rational so long as we consider $h$ which is not a polynomial in some proper subset of $\{x_1,..., x_n\}$. If $h$ is polynomial in a proper subset of $\{x_1,..., x_n\}$, then we may make a linear change of variables $\{z_1,...,z_{\nu}\}$, $\nu < n$, which allows us to consider $\Delta(h(z_1,...,z_{\nu})) \subset \mathbb{R}^{\nu}$, where $\mu_h$ will be finite.


\vspace{1pc}

The quantity $\mu_f$ is used by Adolphson and Sperber \cite{Adolph} to put a lower bound on the $q$-adic valuation ord$_q$ of the number of $\mathbb{F}_q$-rational points on a variety $V$, $N(V)$, over $\mathbb{F}_q$. Namely, let $V = Z(f_1,...,f_m)$ be the vanishing set of $f_1,...,f_n$, where $f_i \in \mathbb{F}_q[x_1,...,x_n]$. If the collection of polynomials $f_1,..., f_m$ is not polynomial in some proper subset of $x_1,..., x_n$, then we have for $f(x_1,..., x_n,y_1,...,y_m) = f_1(x_1,...,x_n)y_1 + \cdots + f_m(x_1,...,x_n)y_m$, \begin{eqnarray*} \textrm{ord}_q (N(V)) \geq  \mu_f - m. \end{eqnarray*}

Note that in the above definitions, the multivariate polynomial $h$ maps the vector space $F^n$ into its base field $F$. However, for the value set problem, we are interested in studying the polynomial vector $f:\mathbb{F}_q^n \longrightarrow \mathbb{F}_q^n$. Fortunately, the definitions we have developed in this section can be extended to polynomial vectors. If we denote the support of $h$ by $\Gamma(h) \defeq \{D_1,...,D_m\}$, then we define $\Delta(f)$ to be the convex closure of $\Gamma(f_1) \cup \cdots \cup \Gamma(f_n) \cup \{(0,...,0)\}$ in $\mathbb{R}^n$.

\section{The Degree Matrix and Comparison of Constants} \label{Degree}

For our multivariate polynomial $h$ as in Section \ref{NewPoly}, we define the $n \times m$ degree matrix of $h$, $D_h \defeq (D_1,...,D_m) \in \mathbb{Z}_{\geq 0}^{n \times m}$. The degree matrix has been used by Cao and his collaborators in \cite{Cao20091778}, \cite{CaoSun}, and \cite{ChenCao} in rational point counting and $p$-adic estimates. In relation to the value set problem, Zan and Cao use the degree matrix in \cite{ZanCao} as a succinct way of keeping track of the exponent vectors $D_j$ that does not explicitly rely on a geometry. Using this, they define the following invariant of $h$. \begin{definition}[Integral dilation factor $\omega_h$] Let $F$ be a field, let $h \in F[x_1,...,x_n]$ be as in equation \eqref{sparse}.  \begin{equation} \omega_h \defeq \textnormal{min}\left\{ \sum_{j=1}^m k_j \left| k_j \in \{0,1,...,q-1\}, \sum_{j=1}^m k_jD_j \in (q-1)\mathbb{N}^n \right. \right\}. \end{equation} \end{definition} This constant can be thought of as the minimal number of exponent vectors (up to $q-1$ duplicates of each) needed to be summed together to reach a lattice point where all coordinates are positive multiples of $q-1$. Again, so long as $h$ is not polynomial in some proper subset of $\{x_1,...,x_n \}$, $\omega_h$ will always exist.

\vspace{1pc}

Though $\omega_f$ and $\mu_{f}$ may seem different by their definitions, a lemma in \cite{WeiCao} gives us that \begin{equation} \mu_f = \textnormal{min}\left\{ \sum_{j=1}^m \alpha_j \left| \alpha_j \in \mathbb{Q}_{\geq 0}, \sum_{j=1}^m \alpha_jD_j \in \mathbb{N}^n \right. \right\}. \end{equation} Intuitively, studying the dilation of $\Delta(f)$ is equivalent to studying linear combinations of the exponent vectors geometrically. Because of similarity, we can use both $D_f$ and $\Delta(f)$ to study $\mu_f$ and $\omega_f$.\footnote{Theorems in Chapter \ref{refine} which are dependent on these constants both use $\Delta(f)$ in the proofs given.}

 In fact, because of this similarity, we have a direct comparison of the two terms proven by \cite{ZanCao}. This, alongside a result of Adolphson and Sperber \cite{Adolph}, gives us the following inequalities: \begin{lemma} $\omega_f \geq \mu_f(q-1) \geq \frac{n(q-1)}{d}.$ \end{lemma} Not only does $\omega_f$ provide a better value set bound for nonpermutation polynomials, but \cite{ZanCao} gives sharp examples which improve previously known univariate bounds. For an illustration of the proof by Adolphson and Sperber in two dimensions, please refer to Figure \ref{polytopefigure} given below.




\begin{figure}[ht!] 
  \centering
\includegraphics[scale=1]{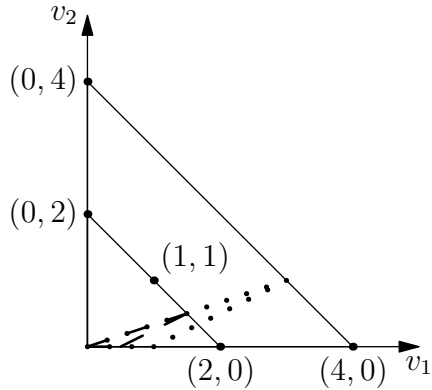}
 \caption{The polytopes of $f(x_1, x_2) = x_1 + x_1^3x_2$ and $h(x_1, x_2) = x_1^4 + x_2^4$, alongside their contractions by $\frac{n}{d} = \frac{2}{4}$. Note that both polynomials are degree 4, $\Delta(f) \cap \mathbb{N}^2 = \{(3,1)\}$, and $\left(\frac{2}{4}\Delta(f)\right) \cap \mathbb{N}^2 = \varnothing$, but $\left(\frac{2}{4}\Delta(h)\right) \cap \mathbb{N}^2 = \{(1,1)\}$. Therefore, $\mu_h = \frac{2}{4} < \mu_f = 1$.} \label{polytopefigure}
\end{figure}

%% file: Lukechapter1.tex
\chapter{Essential Theorems and Concepts} \label{Prelim}

\section{Single Variable Value Set}
\label{Single}

To provide insight towards the proof of our main result, we will investigate upper bounds of $|V_f|$ for the case when $f$ is a single variable polynomial. Parts of this proof will generalize to the multivariate case.

\begin{theorem} \label{SingleVar} Let $f(x) \in \mathbb{F}_q[x]$ be a single variable polynomial of degree $d > 0$. If $|V_f| < q$, then \begin{eqnarray*}|V_f| \leq q - \frac{q-1}{d}.\end{eqnarray*}
\end{theorem}

The proof of this theorem relies on the following definition:

\begin{definition}[The quantity $U(f)$] Let $\mathbb{Z}_q$ denote the ring of $p$-adic integers with uniformizer $p$ and residue field $\mathbb{F}_q$. Also let $\tilde{f}(x) \in \mathbb{Z}_q[x]$ be the lifting of $f$ taking coefficients from the Teichm$\ddot{\textrm{u}}$ller lifting $L_q \subset $ $\mathbb{Z}_q$ of $\mathbb{F}_q$. Then we define $U(f)$ to be the smallest positive integer $k$ such that the sum \begin{eqnarray*} S_k(f) \defeq \sum_{x \in L_q} \tilde{f}(x)^k \not\equiv 0 \textrm{ (mod } pk). \end{eqnarray*}
\end{definition}

By taking into account the following sum,
\begin{eqnarray} \label{charsum}   \sum_{x \in L_q} x^k = \left\{     \begin{array}{ll}       0, & q-1 \nmid k,\\        q-1, & q-1 \mid k, k \neq 0,\\  q, & k = 0,     \end{array}   \right.\end{eqnarray} 
and remembering that we are only summing over a finite number of terms, we have that, for $f$ not identically zero, $\frac{q-1}{d} \leq U(f)$. We also have that if $f$ is a permutation polynomial, then $S_k(f) = S_k(x) = \sum_{x \in L_q} x^k$, implying $U(f) = q-1$. The fact that $U(f)$ exists for all nonpermutation polynomials as well is a corollary of lemma \ref{Ulemma}. Overall, the lemma and the above argument give us that

$$ \frac{q-1}{d} \leq U(f) \leq q-1. $$
Theorem \ref{SingleVar} also follows directly from the lemma \ref{Ulemma}:

\begin{lemma} \label{Ulemma} If $\left|V_f \right| < q$, then \begin{eqnarray*} \left|V_f \right| \leq q - U(f).\end{eqnarray*}
\end{lemma}

The proof of this result is given by Wan, Shiue, and Chen in \cite{WSCsharp}, and their paper also includes more details regarding this lemma.  Mullen, Wan, and Wang \cite{MWWValue} also describe an alternate proof of this lemma presented to them by Lenstra through private communication.

\section{From Single Variable to Multivariable}
\label{From2}

Let $f(x_1,...,x_n)=(f_1(x_1,...,x_n),...,f_n(x_1,...,x_n))$ be a polynomial vector, and note $\textrm{deg } f = \textrm{max}_i\{\textrm{deg } f_i\}$. This maps the vector space $\mathbb{F}_q^n$ to itself. Now, take a basis $e_1,...,e_n$ of $\mathbb{F}_{q^n}$ over $\mathbb{F}_q$. Denote $x = x_1e_1 + \cdots + x_ne_n$ and define \begin{eqnarray*}g(x) \defeq f_1(x_1,...,x_n)e_1+ \cdots +f_n(x_1,...,x_n)e_n. \end{eqnarray*} In this way, we can think of the function $g$ as a non-constant univariate polynomial map from the finite field $\mathbb{F}_{q^n}$ to itself. Even better, we have the equality $|V_f| = |g(\mathbb{F}_{q^n})|$. Therefore, using Lemma~\ref{Ulemma}, we know \begin{eqnarray*} \textrm{if } |V_f| < q^n, \textrm{then } |V_f| \leq q^n - U(g),\end{eqnarray*} where $g$ is viewed as a univariate polynomial.

\vspace{1pc}

Unfortunately, as a univariate polynomial, we do not have good control of the univariate degree of $g$ in relation to the multivariate degree of $f$. Even if one were to construct a closed form for $g(x)$ using methods such as Lagrange Interpolation, the degree of $g$ would likely be high enough as to make the resulting upper bound on $|V_f|$ trivial. Because of these issues with the degree of $g$, we cannot use the bounds from the previous section directly, and must rely on another method to bound $U(g)$.

\vspace{1pc}

Previously, we introduced $g(x)$ as a univariate polynomial. However, using a basis $e_1,...,e_n$ of $\mathbb{F}_{q^n}$ over $\mathbb{F}_q$ as before, we can also define a multivariate polynomial \begin{eqnarray*}g(x_1,...,x_n) \defeq f_1(x_1,...,x_n)e_1+ \cdots +f_n(x_1,...,x_n)e_n \end{eqnarray*} mapping the vector space $\mathbb{F}_q^n$ into the field $\mathbb{F}_{q^n}$. In this sense, $g$ as a multivariate polynomial shares some important properties with $f$ as a polynomial vector, such as the fact that deg($g$) = $\textrm{max}_i\{\textrm{deg } f_i\}$. Whereas the paper by Mullen, Wan, and Wang determine a bound for $U(g)$ relying on the multivariate degree of $f$, in this paper we will use the Newton polytope of the multivariate polynomial $g(x_1,...,x_n)$ to improve upon these bounds. With this in mind, we define $\Delta(f) \defeq \Delta(g(x_1,...,x_n))$, $\mu_f \defeq \mu_{g(x_1,...,x_n)}$, and prove the original polytope bound in \cite{Mypaper}.   

\section{Restatement of First Polytope Bound and Proof}
\label{Multi}

\begin{theorem} \label{mubound} Let $f(x_1,...,x_n)=(f_1(x_1,...,x_n),...,f_n(x_1,...,x_n))$ be a polynomial vector over the vector space $\mathbb{F}_q^n$. If $|V_f| < q^n$, then \begin{eqnarray*}|V_f| \leq q^n - \textnormal{min}\{q,\hspace{.25pc} \mu_f(q-1)\}.\end{eqnarray*}
\end{theorem}

\begin{proof} First, construct $g$ from our polynomial vector $f$, as we did in Section \ref{From2}. Viewing $g$ as a univariate polynomial $g(x)$, we are allowed to apply Lemma~\ref{Ulemma} to bound $|V_f|$ using $U(g)$. We then consider $g$ as multivariate $g(x_1,...,x_n)$, which allows us to define $\Delta(g)$ and $\mu_g$. Noting that $\Delta(f) = \Delta(g)$ and $\mu_f = \mu_g$ by our definition in Section \ref{From2}, it suffices to prove the following lemma on $U(g)$:
\end{proof}

\begin{lemma} \label{Uglemma} $U(g) \geq \textnormal{min}\{\mu_f(q-1),\hspace{.25pc} q\}.$
\end{lemma}

\begin{proof} Assume the coefficients of $g(x_1,...,x_n)$ are lifted to characteristic zero over $L_{q^n}$, our Teichm$\ddot{\textrm{u}}$ller lifting of $\mathbb{F}_{q^n}$. Remember that $U(g)$ is defined over univariate polynomials to be the smallest positive integer $k$ such that \begin{eqnarray*} S_k(g) \defeq \sum_{x \in L_{q^n}} {g}(x)^k \not\equiv 0 \textrm{ (mod } pk). \end{eqnarray*} However, using $x = x_1e_1 + \cdots + x_ne_n$ as in Section \ref{From2}, we can rewrite $S_k(g)$ in terms of multivariate $g(x_1,...,x_n)$. This means $U(g)$ is the smallest positive integer $k$ such that  \begin{eqnarray*} S_k(g) = \sum_{(x_1,...,x_n) \in L_q^n} {g}(x_1,...,x_n)^k \not\equiv 0 \textrm{ (mod } pk). \end{eqnarray*}

Let $k \in \mathbb{Z}_{>0}$ be such that $k < \textnormal{min}\{\mu_f(q-1),\hspace{.25pc} q\}.$ Expand $g(x_1,...,x_n)^k=\sum_{j=1}^m a_jX^{V_j}$ as a polynomial in the $n$ variables $x_1,...,x_n$ (see $(\ref{polynota})$). Since $S_k(g)$ is a finite sum, it can be broken up over the monomials of $g(x_1,...,x_n)^k$. Therefore, it suffices to prove \begin{eqnarray} \label{xvjsum} \sum_{(x_1,...,x_n) \in L_q^n}X^{V_j} \equiv 0 \textrm{ (mod } pk), \hspace{.25pc} 1 \leq j \leq m, \, m = \textnormal{\# monomials of } g^k(x_1,...,x_n).\end{eqnarray} If we denote $\ell_j \defeq \#\{v_{ij}, \hspace{.25 pc} 1 \leq i \leq n| v_{ij} \neq 0 \}$, i.e. $\ell_j$ denotes the number of nonzero $v_{ij}$'s with $1 \leq i \leq n$, then we have exactly $n - \ell_j$ zero $v_{ij}$'s, implying that \begin{eqnarray*} \sum_{(x_1,...,x_n) \in L_q^n}X^{V_j} \equiv 0 \textrm{ (mod } q^{n-\ell_j}). \end{eqnarray*} Now let $v_p$ denote the $p$-adic valuation satisfying $v_p(p)=1$. If the inequality \begin{eqnarray*} v_p(q)(n-\ell_j) \geq 1 + v_p(k) \end{eqnarray*} is satisfied, then (\ref{xvjsum}) is true and we are done.

\vspace{1 pc}

Considering $X^{V_j} = x_1^{v_{1j}}...x_n^{v_{nj}}$, the sum on the left side is identically zero if one of the $v_{ij}$ is not divisible by $q-1$ (see (\ref{charsum})). Thus, we shall assume that all $v_{ij}$'s are divisible by $q-1$ (Otherwise (\ref{xvjsum}) is satisfied and we are done without even using our inequality on $k$). 

\vspace{1pc}

Now, the lattice points of $g$ are contained within $\Delta(g)$ by definition, and this implies our lattice points $V_j$ of $g^k$ are contained within $k\Delta(g)$, the dilation of the polytope $\Delta(g)$ by $k$. But since $(q-1) \mid v_{ij}$, we have that $V_j \in (q-1)\mathbb{Z}^n_{\geq 0}$ as well. 

\vspace{1pc}

If we further assume that $V_{j}$ has no zero coordinates, i.e. $\ell_j = n$, this implies \begin{eqnarray*}\left(\frac{k}{q-1}\Delta(g)\right) \cap \mathbb{Z}^n_{>0} \neq \varnothing.\end{eqnarray*} This statement tells us, by the definition of $\mu_f$, that $\frac{k}{q-1} \geq \mu_f$. In other words, \begin{eqnarray*} k \geq \mu_f(q-1). \end{eqnarray*} This contradicts our assumption that $k < \textnormal{min}\{\mu_f(q-1),\hspace{.25pc} q\} \leq \mu_f(q-1).$

\vspace{1pc}

Therefore, when $k < \textnormal{min}\{\mu_f(q-1),\hspace{.25pc} q\},$ we have that $\ell_j < n$, and $n - \ell_j > 0.$ This case, since $k < q,$ gives us $q \nmid k,$ and \begin{eqnarray*} 1 + v_p(k) \leq v_p(q) \leq v_p(q)(n - \ell_j). \end{eqnarray*} This implies that \begin{eqnarray*} S_k(g) \equiv 0 \textrm{ (mod } q^{n-\ell_j}) \equiv 0 \textrm{ (mod } p^{1 + v_p(k)}) \equiv 0 \textrm{ (mod } pk) \end{eqnarray*} and we are done. Lemma~\ref{Uglemma} and the main result of \cite{Mypaper} are proved. \end{proof}





%% file: Chapter2refining.tex
\chapter{Refining Cardinality Bounds} \label{refine}

\section{A Method to Improve Prior Proofs} \label{newmethod}

One of the major limitations of the use of $p$-adic liftings in the proof of Lemma \ref{Uglemma} is that our assumption only allowed us to show $S_k(f) \equiv 0$ (mod $q$). Indeed, if we immediately split $S_k(f)$ amongst the monomials of the multivariate polynomial $g(x_1,...,x_n)^k$, we lose much of the structure and divisibility of each term. Therefore, we will manipulate our summand to leverage a larger $p$-adic valuation before splitting it into monomials. To do this, we need the following lemma:

\begin{lemma}\label{ppowered} Let $x_1,...,x_n$ be in a commutative ring $R$, and let $e \in \mathbb{N}$. Then \begin{equation*} \begin{split} (x_1 + \cdots + x_n)^{p^e} = x_1^{p^e} + \cdots + x_n^{p^e} & + ph_1(x_1^{p^{e-1}},...,x_n^{p^{e-1}}) + p^2h_2(x_1^{p^{e-2}},...,x_n^{p^{e-2}}) \\ & + \cdots + p^eh_e(x_1,...,x_n) \end{split}\end{equation*} where $h_t(x_1^{p^{e-t}},...,x_n^{p^{e-t}}) \in R[x_1,...,x_n]$ is such that \textnormal{deg} $h_t(x_1,...,x_n) = p^t$.
\end{lemma}

\begin{proof} We use the multinomial theorem on the left hand side of the above equation.

\begin{equation} \label{basey} (x_1 + \cdots + x_n)^{p^e} = x_1^{p^e} + \cdots + x_n^{p^e} + \sum_{\substack{a_1 + \cdots + a_n =p^e \\ a_1 \neq p^e,..., a_n \neq p^e}} \binom{p^e}{a_1,...,a_n} x_1^{a_1}...x_n^{a_n}. \end{equation}

For simplicity of notation, let \begin{equation*} A = \binom{p^e}{a_1,...,a_n} x_1^{a_1}...x_n^{a_n}. \end{equation*}

Then the sum in (\ref{basey}) can be split as follows: \begin{equation*} \begin{split}
\sum_{a_1 + \cdots + a_n =p^e} A = x_1^{p^{e}} + \cdots + x_n^{p^{e}} & + \sum_{\substack{a_1 + \cdots + a_n =p^e \\ p^{e-1} || (a_1,...,a_n)}} A + \sum_{\substack{a_1 + \cdots + a_n =p^e \\ p^{e-2} || (a_1,...,a_n)}} A
\\ + \cdots & + \sum_{\substack{a_1 + \cdots + a_n =p^e \\ p || (a_1,...,a_n)}} A + \sum_{\substack{a_1 + \cdots + a_n =p^e \\  p \, \nmid a_{\epsilon} \, \textnormal{for some } {\epsilon}}} A. 
\end{split}
\end{equation*}
Now, let  \begin{equation*} \sigma_t = \sum_{\substack{a_1 + \cdots + a_n =p^e \\ p^{e-t} || (a_1,...,a_n)}} A, \hspace{.5pc} 1 \leq t \leq e. \end{equation*} If we can show for $1 \leq t \leq e$ that $\sigma_t$ has the form $p^t h_t(x_1^{p^{e-t}},...,x_n^{p^{e-t}})$ with deg $h_t(x_1,...,x_n) = p^t$, then the proof is done.

\vspace{1pc}
Notice that the summand $A$ always has degree $a_1 + \cdots + a_n = p^e$, which means deg $\sigma_t = p^e$. Since $p^{e-t} | a_{\epsilon}$ for all ${\epsilon}$ between 1 and $n$, we know that $\sigma_t$ has the form $\tau_t(x_1^{p^{e-t}},...,x_n^{p^{e-t}}) \in R[x_1,...,x_n]$ and deg $\tau_t(x_1,...,x_n) = \frac{p^e}{p^{e-t}} = p^t$.

\vspace{1pc}
The fact that $p^t | \binom{p^e}{a_1,...,a_n}$ under the conditions that $p^{e-t} || (a_1,...,a_n)$ has an elegant proof by Singmaster in \cite{Singmaster}. Therefore, we have that $p^t | \tau_t(x_1^{p^{e-t}},...,x_n^{p^{e-t}})$. This tells us $\sigma_t$ has the form $p^t h_t(x_1^{p^{e-t}},...,x_n^{p^{e-t}})$ with deg $h_t(x_1,...,x_n) = p^t$, and thus the lemma is proved.

\end{proof}

Let $f$ be a polynomial map over $\mathbb{F}_q^n,$ char $\mathbb{F}_q = p.$ Also let $e_1, ..., e_n$ be a basis of the field $\mathbb{F}_{q^n}$ over $\mathbb{F}_q$, and let $x = x_1e_1 + \cdots + x_ne_n$ as before in Section \ref{From2}, allowing for the identification of a polynomial map $f(x_1,..,x_n) = ((f_1(x_1,...,x_n),...,(f_n(x_1,...,x_n))$ with the multivariate polynomial $f(x_1,...,x_n) = f_1(x_1,...,x_n)e_1 + \cdots + f_n(x_1,...,x_n)e_n$ or the univariate polynomial $f(x)$. Also let $S_k(f)$ and $U(f)$ be as in Section \ref{SingleVar}. To improve upon the $p$-adic lifting method, we will apply Lemma \ref{ppowered} to $f(x_1,...,x_n)^k$, split $S_k(f)$ amongst these polynomials, and then split the summand polynomials further into monomials.

\vspace{1pc}

Write $k = p^ek_1$ with $p \nmid k_1.$ For simplicity of notation, assume $f$ has already been lifted with coefficients in $L_{q^n}$. Then by Lemma \ref{ppowered}, there exists polynomials $F_0,..., F_e \in \mathbb{F}_{q^n}[x_1,...,x_n]$ such that \begin{align*} \left(f_1(x_1,...,x_n)e_1 + \cdots + f_n(x_1,...,x_n)e_n\right)^{p^e} & = F_0(x_1^{p^e},...,x_n^{p^e}) + pF_1(x_1^{p^{e-1}},...,x_n^{p^{e-1}}) + \cdots \\ & + p^eF_e(x_1,...,x_n), \end{align*} where deg $F_t(x_1, ..., x_n) \leq dp^t.$ This means that 

\begin{align} \left(f_1(x_1,...,x_n)e_1 + \cdots + f_n(x_1,...,x_n)\right)^{k} & = \left(F_0(x_1^{p^e},...,x_n^{p^e}) + pF_1(x_1^{p^{e-1}},...,x_n^{p^{e-1}}) \right. + \cdots \notag \\ & + p^eF_e(x_1,...,x_n) \Bigr)^{k_1} \notag \\ = \sum_{b_0 + \cdots + b_e =k_1} \binom{k_1}{b_0,...,b_e} p^{b_1 + 2b_2 + \cdots + eb_e} &F_0(x_1^{p^e},...,x_n^{p^e})^{b_0} \cdots F_e(x_1,...,x_n)^{b_e}. \notag \end{align}

\vspace{1pc}

Now for fixed $b_0,..., b_e,$ let $\lambda$ be the positive integer such that $b_{\lambda} \neq 0, b_{\lambda + 1} = \cdots = b_e = 0,$ and let $y_i = x_i^{p^{e-\lambda}}$. This means we can reduce the power and degree of our summand polynomials in the following way: \begin{equation} \label{lambdasub} F_0(x_1^{p^e},...,x_n^{p^e})^{b_0} \cdots F_{\lambda}(x_1^{p^{e-\lambda}},...,x_n^{p^{e-\lambda}})^{b_{\lambda}} = F_0(y_1^{p^{\lambda}},...,y_n^{p^{\lambda}})^{b_0} \cdots F_{\lambda}(y_1,...,y_n)^{b_{\lambda}}. \end{equation} Note that each term may have a different substitution, but we may split $S_k(f)$ amongst each summand to bound the $p$-divisibility of the entire sum. Using the reduction of $f(x_1,..,x_n)^k$ to \eqref{lambdasub}, we are given sums of the form \begin{equation} \label{lambdareducedsum} p^{b_1 + 2b_2 + \cdots + \lambda b_{\lambda}}\sum_{y_1,...,y_n \in L_q}F_0(y_1^{p^{\lambda}},...,y_n^{p^{\lambda}})^{b_0} \cdots F_{\lambda}(y_1,...,y_n)^{b_{\lambda}}. \end{equation} Now the fact that $b_{\lambda} \neq 0$ tells us this sum is divisible by $p^{\lambda}$, i.e. $S_k(f) \equiv 0 \, (\textnormal{mod } p^{\lambda})$. From here, we must further split this summand product into monomials and determine the $p$-divisibility of the smaller sums. Let \begin{equation*} F_0(y_1^{p^{\lambda}},...,y_n^{p^{\lambda}})^{b_0} \cdots F_{\lambda}(y_1,...,y_n)^{b_{\lambda}} = \sum_{j=1}^m c_jY^{W_j}, \hspace{1 pc} c_j \in \mathbb{F}_{q^n}^{*}, \end{equation*} where \begin{equation} \label{Mononota} W_j = (w_{1j},...,w_{nj}) \in \mathbb{Z}_{\geq 0}^n, \hspace{1 pc} Y^{W_j} = y_1^{w_{1j}}...y_n^{w_{nj}}. \end{equation} This allows the sum in \eqref{lambdareducedsum}, and ultimately $S_k(f)$, to be split among the monomials in \eqref{Mononota} into sums of the form \begin{equation} \label{Monomproduct} p^{b_1 + 2b_2 + \cdots + \lambda b_{\lambda}}\sum_{y_1,...,y_n \in L_q}c_j Y^{W_j} = c_j p^{b_1 + 2b_2 + \cdots + \lambda b_{\lambda}} \prod_{i = 1}^n \sum_{y_i \in L_q} y_1^{w_{ij}}. \end{equation} 

\vspace{1pc}

Let $C$ be an upper bound on $k$, i.e. $ k < C$  ($C$ will depend on which of the theorems in the following sections we are proving). Our goal is to show $C \leq U(f)$ and therefore come up with a nicer bound on $|V_f|$ (thanks to Lemma \ref{Ulemma}). Let $v_p$ be the $p$-adic valuation with $v_p(p) = 1$, and let $\ell_j$ be the number of nonzero entries of $W_j$. We can accomplish our goal by showing the sum in \eqref{Monomproduct} is congruent to zero mod $p^{\lambda}q^{n- \ell_j},$ and that $v_p(p^{\lambda}q^{n- \ell_j}) \geq v_p(pk)$, i.e. $$ \lambda + (n- \ell_j)v_p(q) \geq e + 1.$$ If this holds true for all monomials, then $S_k(f) \equiv 0$ mod $pk$ and $k \leq U(f)$.

\vspace{1pc}

Now the sum in \eqref{Monomproduct} equals zero if one of the $w_{ij}$'s is not divisible by $q-1$, so all that is left is to consider the case when $q-1 | w_{ij}$ for all $i$.  In this case, since $b_{\lambda} \neq 0$, and since $\ell_j$ is the number of nonzero $w_{ij}$, we have $n - \ell_j$ zero terms, which tells us $$p^{b_1 + 2b_2 + \cdots + \lambda b_{\lambda}}\sum_{y_1,...,y_n \in L_q}c_j Y^{W_j} \equiv 0 \, (\textnormal{mod } p^{\lambda}q^{n - \ell_j}). $$

\vspace{1pc}

The above substitution method allows us to refine the recently published results mentioned in Section \ref{recent}, whose proofs simply used the monomials of $f(x_1,...,x_n)^k$ directly. These proofs required that $k < q$ to bound the value set, but our proofs do not. The next few sections will show how the added structure our method provides tighter upper bounds on the cardinality of the value set.

\section{Alternate Degree Bound Proof}

\begin{theorem} \label{Kostersdegree} Let $f$ be a polynomial map with $f: \mathbb{F}_q^n \longrightarrow \mathbb{F}_q^n,$ char $\mathbb{F}_q = p,$ $$f(x_1,...,x_n) = f_1(x_1,...,x_n)e_1 + \cdots + f_n(x_1,...,x_n),$$ and $d$ = max$_i$ deg $f_i$. If $|V_f| < q^n$, then $$V_f \leq q - \frac{n(q-1)}{d}.$$
\end{theorem}

Note that this theorem was proven by Kosters in \cite{Kosters}, but we provide an alternate proof using the method outlined in Section \ref{newmethod}.

\begin{proof} If we can show that, for $1 \leq k < \frac{n(q-1)}{d}$ and $k = p^ek_1$ with $p \nmid k_1$,  \begin{equation*} S_k(f) \defeq \sum_{x \in L_{q^n}} \tilde{f}(x)^k  = \sum_{x_1,...,x_n \in L_{q}} \left(\tilde{f}_1(x_1,...,x_n)\tilde{e}_1 + \cdots + \tilde{f}_n(x_1,...,x_n)\tilde{e}_n\right)^k \equiv 0 \, (\textnormal{mod } pk), \end{equation*} then $U(f) \geq \frac{n(q-1)}{d}$ and we are done by Lemma \ref{Ulemma}.

\vspace{1pc}

For simplicity of notation, assume $f$ is already lifted to characteristic zero over $L_{q^n}$. Split $S_k(f)$ into sums of the form \eqref{lambdareducedsum}. Notice that, by our substitution and Lemma \ref{ppowered}, the degree of the summand $F_0(y_1^{p^{\lambda}},...,y_n^{p^{\lambda}})^{b_0} \cdots F_{\lambda}(y_1,...,y_n)^{b_{\lambda}}$ in $(y_1,...,y_n)$ is bounded above by $dp^{\lambda}a_0 + dp^1p^{\lambda - 1} + \cdots + dp^{\lambda}a_{\lambda} = dp^{\lambda}k_1.$ When we further split these sums into sums of the form \eqref{Monomproduct}, we have that \begin{equation*} p^{b_1 + 2b_2 + \cdots + \lambda b_{\lambda}}\sum_{y_1,...,y_n \in L_q}c_j Y^{W_j} \equiv 0 \, (\textnormal{mod } p^{\lambda}q^{n - \ell_j}). \end{equation*} Since this sum equals 0 if one of the $w_{ij}$'s is not divisible by $q-1$, assume $q-1 | w_{ij}$ for all $i$. Using our degree bound we have that \begin{equation*} (q-1) \ell_j \leq w_{1j} + \cdots w_{nj} \leq dp^{\lambda}k_1, \end{equation*} or $\ell_j \leq \left\lfloor \frac{dp^{\lambda}k_1}{q-1} \right\rfloor$. If we can show that \begin{equation*} v_p\left(p^\lambda q^{n - \left\lfloor \frac{dk_1p^{\lambda}}{q-1}\right\rfloor}\right) \geq v_p(pk), \end{equation*} Then we are done. In other words, we must show \begin{equation*} \lambda + \left(n - \left\lfloor{\frac{dk_1p^{\lambda}}{q-1}}\right\rfloor \right)v_p(q) \geq e + 1. \end{equation*}

\vspace{1pc}

Mullen, Wan, and Wang \cite{MWWValue} proved a similar inequality, \begin{equation*}\left(n - \left\lfloor{\frac{dk_1p^e}{q-1}}\right\rfloor \right)v_p(q) \geq e + 1. \end{equation*} However, their proof only holds in the cases when: \begin{enumerate}
\item $n \leq d$
\item $n > d$ and $k < q$.
\end{enumerate}
Since their inequality implies ours, we may assume that $n > d$ and $k \geq q$. Let $r = e - \lambda$. Then it suffices to show that \begin{equation*} \left(n - \left\lfloor{\frac{dk}{p^r(q-1)}}\right\rfloor \right)v_p(q) \geq r + 1, \hspace{.5pc} 0 \leq r \leq e. \end{equation*} Note that $k < \frac{n(q-1)}{d} \leq q^n$ implies that $\frac{dk}{q-1} < n$, which is equivalent to $\frac{dk}{p^r(q-1)} < \frac{n}{p^r}.$ 
\vspace{1pc}

In the case that $n \leq p^r$, we have that $\frac{dk}{p^r(q-1)} < 1$, which implies \begin{equation*} \left(n - \left\lfloor{\frac{dk}{p^r(q-1)}}\right\rfloor \right)v_p(q) = nv_p(q) > v_p(k) = e \geq r. \end{equation*} In other words, \begin{equation*} \left(n - \left\lfloor{\frac{dk}{p^r(q-1)}}\right\rfloor \right)v_p(q) \geq r + 1. \end{equation*} Now let us examine the case when $n > p^r$. For $r = 0$, \begin{equation*} \left(n - \left\lfloor{\frac{dk}{(q-1)}}\right\rfloor \right)v_p(q) \geq v_p(q) \geq 1. \end{equation*} For $r = 1$ and $p = 2$ (implying $n > 2$), \begin{equation*} \left(n - \left\lfloor{\frac{dk}{2(q-1)}}\right\rfloor \right)v_p(q) \geq \left(n - \left\lfloor{\frac{n}{2}}\right\rfloor \right)v_p(q) = \left\lceil{\frac{n}{2}}\right\rceil v_p(q) \geq 2v_p(q) \geq 2. \end{equation*} And finally, when $r \geq 1$, \begin{equation*} \begin{split}\left(n - \left\lfloor{\frac{dk}{p^r(q-1)}}\right\rfloor \right)v_p(q) & \geq \left(n - \left\lfloor{\frac{n}{p^r}}\right\rfloor \right)v_p(q) \geq \left(\frac{n(p^r-1)}{p^r}\right)v_p(q) \\ & \geq \left(\frac{(p^r + 1)(p^r-1)}{p^r}\right)v_p(q) = \left(p^r - \frac{1}{p^r}\right)v_p(q). \end{split}\end{equation*} Note that $p^r - \frac{1}{p^r} \geq r + 1$ for all $r \geq 1$ except for when $r = 1$ and $p = 2$ simultaneously.

\end{proof}

\section{Improved Newton Polytope Bound}

\begin{theorem} \label{mubound2} Let $f(x_1,...,x_n)=(f_1(x_1,...,x_n),...,f_n(x_1,...,x_n))$ be a polynomial vector over the vector space $\mathbb{F}_q^n$. Without loss of generality, suppose $f$ is not polynomial in some subset of $\{x_1,...,x_n\}$. Let $\Delta(f)$ be the Newton polytope of $f$, and let $\mu_f$ be the minimal dilation factor associated with $\Delta(f)$. If $|V_f| < q^n$, then \begin{align*}|V_f| \leq q^n - \mu_f(q-1).\end{align*}
\end{theorem}

\begin{proof}  If we can show that, for $1 \leq k < \mu_f(q-1)$ and $k = p^ek_1$ with $p \nmid k_1$, \begin{equation*} S_k(f) \defeq \sum_{x \in L_{q^n}} \tilde{f}(x)^k  = \sum_{x_1,...,x_n \in L_{q}} \left(\tilde{f}_1(x_1,...,x_n)\tilde{e}_1 + \cdots + \tilde{f}_n(x_1,...,x_n)\tilde{e}_n\right)^k \equiv 0 \, (\textnormal{mod } pk), \end{equation*} then $U(f) \geq \mu_f(q-1)$ and we are done by Lemma \ref{Ulemma}.

\vspace{1pc}

For simplicity of notation, assume $f$ is already lifted to characteristic zero over $L_{q^n}$. Split $S_k(f)$ into sums of the form \eqref{lambdareducedsum}. Notice that, by our substitution and Lemma \ref{ppowered}, the exponent vectors of the monomials of the product $F_0(y_1^{p^{\lambda}},...,y_n^{p^{\lambda}})^{b_0} \cdots F_{\lambda}(y_1,...,y_n)^{b_{\lambda}}$ are contained in $\frac{k}{p^{e-\lambda}}\Delta(f)$. When we further split these sums into sums of the form \eqref{Monomproduct}, we have that \begin{equation*} p^{b_1 + 2b_2 + \cdots + \lambda b_{\lambda}}\sum_{y_1,...,y_n \in L_q}c_j Y^{W_j} \equiv 0 \, (\textnormal{mod } p^{\lambda}q^{n - \ell_j}). \end{equation*} Since this sum equals 0 if one of the $w_{ij}$'s is not divisible by $q-1$, assume $q-1 | w_{ij}$ for all $i$. By this assumption, we have that \begin{equation} \label{inintersection} W_j \in \frac{k}{p^{e-\lambda}}\Delta(f) \cap (q-1)\mathbb{Z}_{\geq 0}^n. \end{equation}

\vspace{1pc}

To further develop this proof, we require additional terminology. \begin{definition}[The quantity $\gamma$] \label{gammaray} \begin{equation*} \gamma \defeq \textrm{min}\left\{ |S| \left| S \subseteq \{W_1,..., W_m\}, \sum_{W_j \in S} W_j \in \mathbb{N}^n \right.\right\}. \end{equation*} \end{definition} In other words, $\gamma$ is the size of smallest subset of the exponent vectors, $\{W_1,..., W_m\}$, such that the sum of its elements lie in $\mathbb{N}^n$. Since $f(x_1,...,x_n)$ is not polynomial in some proper subset of $\{x_1,...,x_n\}$, we have that the polynomials $F_0(y_1^{p^{\lambda}},...,y_n^{p^{\lambda}}), ..., F_{\lambda}(y_1,...,y_n)$ are not either. This means $\gamma$ will exist. Also, assume without loss of generality that $W_1,...,W_{\gamma}$ satisfy the sum property of $\gamma$, i.e. \begin{equation*} W_1 + \cdots + W_{\gamma} \in \mathbb{N}^n. \end{equation*} Using this and \eqref{inintersection}, we have that \begin{equation} \label{suminintersection} W_1 + \cdots + W_{\gamma} \in \frac{\gamma k}{p^{e-\lambda}}\Delta(f) \cap (q-1)\mathbb{N}^n, \end{equation} which means that $\mu_f \leq \frac{\gamma k}{p^{e-\lambda}(q-1)}.$ Reorganizing this, and using our assumption on $k$ at the beginning of the proof, we have $\frac{p^{e-\lambda}}{\gamma}\mu_f(q-1) \leq k < \mu_f(q-1)$, or $p^{e-\lambda} < \gamma$. To make use of this inequality, we have the following lemma:

\begin{lemma} \label{gammabound} For all integers $1 \leq j \leq m$, we have 
\begin{equation*} \gamma -1 \leq n - \ell_j. \end{equation*}
\end{lemma}

\begin{proof} Let $W_u$ be such that $\ell_u \geq \ell_j$ for all $1 \leq j \leq m$. If $\ell_u = n$, then $\gamma = 1$ and we are done. If not, $W_u$ has $n - \ell_u$ components which are zero and we can pick elements from $\{W_1,...,W_{u-1},W_{u+1},...,W_m\}$ to add to $W_u$ until the resulting sum is an element of $\mathbb{N}^n$. This implies it is possible to pick $n - \ell_u + 1$ vectors from $\{W_1,...,W_m\}$ whose sum will lie in $\mathbb{N}^n$. By the definition of $\gamma$, we must have $\gamma \leq n - \ell_u + 1$. But by our assumption on $W_u$, this means that $\gamma - 1 \leq n - \ell_j$ for all $j$. \end{proof}

\vspace{1pc}

With the help of Lemma \ref{gammabound} and \eqref{suminintersection}, we have that $p^{e-\lambda} \leq \gamma - 1 \leq n - \ell_j$. If we can show that \begin{equation*} v_p\left(p^\lambda q^{p^{e-\lambda}}\right) \geq v_p(pk), \end{equation*} Then $S_k(f) \equiv 0$ mod $(pk)$ and we are done. In other words, if $r = e - \lambda$ we must show \begin{equation}\label{pr} p^{r}v_p(q) \geq r + 1. \end{equation} Fortunately, this is true for all primes $p$ and all positive integers $r$.

\end{proof}

\section{Improved Integral Dilation Bound}

\begin{theorem} \label{omegabound} Let $f(x_1,...,x_n)=(f_1(x_1,...,x_n),...,f_n(x_1,...,x_n))$ be a polynomial vector over the vector space $\mathbb{F}_q^n$. Without loss of generality, suppose $f$ is not polynomial in some subset of $\{x_1,...,x_n\}$. Let $\omega_f$ be the integral dilation factor associated with $\Delta(f)$. If $|V_f| < q^n$, then \begin{align*}|V_f| \leq q^n - \omega_f.\end{align*}
\end{theorem}

\begin{proof} If we can show that, for $1 \leq k < \omega_f$ and $k = p^ek_1$ with $p \nmid k_1$,  \begin{equation*} S_k(f) \defeq \sum_{x \in L_{q^n}} \tilde{f}(x)^k  = \sum_{x_1,...,x_n \in L_{q}} \left(\tilde{f}_1(x_1,...,x_n)\tilde{e}_1 + \cdots + \tilde{f}_n(x_1,...,x_n)\tilde{e}_n\right)^k \equiv 0 \, (\textnormal{mod } pk), \end{equation*} then $U(f) \geq \omega_f$ and we are done by Lemma \ref{Ulemma}.

\vspace{1pc}

For simplicity of notation, assume $f$ is already lifted to characteristic zero over $L_{q^n}$. Split $S_k(f)$ into sums of the form \eqref{lambdareducedsum}. Notice that, by our substitution and Lemma \ref{ppowered}, the exponent vectors of the monomials of the product $F_0(y_1^{p^{\lambda}},...,y_n^{p^{\lambda}})^{b_0} \cdots F_{\lambda}(y_1,...,y_n)^{b_{\lambda}}$ are contained in $\frac{k}{p^{e-\lambda}}\Delta(f)$. When we further split these sums into sums of the form \eqref{Monomproduct}, we have that \begin{equation*} p^{b_1 + 2b_2 + \cdots + \lambda b_{\lambda}}\sum_{y_1,...,y_n \in L_q}c_j Y^{W_j} \equiv 0 \, (\textnormal{mod } p^{\lambda}q^{n - \ell_j}). \end{equation*} Since this sum equals 0 if one of the $w_{ij}$'s is not divisible by $q-1$, assume $q-1 | w_{ij}$ for all $i$. By this assumption, we have that \begin{equation*} W_j \in \frac{k}{p^{e-\lambda}}\Delta(f) \cap (q-1)\mathbb{Z}_{\geq 0}^n. \end{equation*}

\vspace{1pc}

Now let $\gamma$ be as in Definition \ref{gammaray} and, WLOG, let $W_1 + \cdots + W_{\gamma} \in \mathbb{N}^n$. Then by the definition of $\gamma$, we have \begin{equation*} W_1 + \cdots + W_{\gamma} \in \frac{\gamma k}{p^{e-\lambda}}\Delta(f) \cap (q-1)\mathbb{N}^n, \end{equation*} and $\omega_f \leq \frac{\gamma k}{p^{e-\lambda}}$. Using this and Lemma \ref{gammabound}, we have $\frac{p^{e-\lambda}}{\gamma}\omega_f \leq k < \omega_f$, or $p^{e-\lambda} \leq \gamma - 1 \leq n - \ell_j$. By this inequality and \eqref{pr}, we are done.



\end{proof}

%% file: Conclusion.tex
\chapter{Conclusion}

\section{Analysis of Cardinality Bounds}

Each of the bounds given in Chapter \ref{refine} are sharp. Let $N(x_1,...,x_{n-1})$ be the field norm of $\mathbb{F}_{q^{n-1}}$ over $\mathbb{F}_q$. Kosters \cite{Kosters} illustrates that Theorem \ref{Kostersdegree} is sharp using the map $f(x_1,..., x_n) = (x_1, x_2,..., N(x_1,...,x_{n-1}) x_n)$. Based on this example, we give the following sharp example for Theorems \ref{mubound2} and \ref{omegabound}. Let $h(x_1, x_2,..., x_n) = (x_1, x_2,..., N(x_1,...,x_{n-1})^a x_n)$ with $a$ in $\mathbb{N}$. Because $N(x_1,...,x_{n-1})$ is a polynomial containing the monomials $x_1^{n-1},...,x_{n-1}^{n-1}$ with nonzero coefficients, we have that $(a,a,...,a,1) \in \Delta(h)$. This explicitly tells us $\Delta(h) \cap \mathbb{N}^n \neq \varnothing$. We also have for all $V = (v_1,...,v_n) \in \Delta(h) \cap \mathbb{N}^n$, $v_n = 1$. This implies $\mu_h = 1$ and $\omega_h = q-1$. In addition, since the preimage $N^{-1}(0) = \{(0,...,0)\}$, we are given $|V_h| = q^n - (q-1)$. This example highlights the flexibility granted by the use of constants derived from the Newton polytope, since deg $h = a(n-1) + 1$ does not allow for a sharp cardinality bound. This flexibility also gives us more freedom to make substitutions when generating more sharp examples. If $z_1(x),...,z_{n-1}(x)$ are univariate permutation polynomials in $\mathbb{F}_q[x]$, then the maps $g(x_1,...,x_n) = \left(z_1(x_1),...,z_{n-1}(x_{n-1}), N(x_1,...,x_{n-1})^a x_n\right)$ and $h(z_1(x_1),...,z_{n-1}(x_{n-1}), x_n)$ will share the same constants and value set cardinality as $h(x_1,..., x_n)$.

\vspace{1pc}

Using the constant $\omega_f$ also has an advantage when determining bounds on univariate value sets. In this case, since $n = 1$, we have that $\mu_f = \frac{1}{\textnormal{deg } f}$ for all $f \in \mathbb{F}_q[x]$, but Zan and Cao \cite{ZanCao} give a sharp example which improves upon this for $\omega_f$. If $f(x) = x^7 + ax \in \mathbb{F}_{19}[x]$ with $a \neq 0, 4, 5, 8, 16, 17$, then it is easy to check that $\omega_f = 6$, $|V_f| = 13 = 19 - \omega_f < 19 - \left\lceil \frac{1}{7}(18) \right\rceil = 16$.

\vspace{1pc}

Note that, in general, it is not immediately clear how large of an improvement the strongest bound in Theorem \ref{omegabound} provides over our bounds in Theorems \ref{Kostersdegree} and \ref{mubound2}. I have addressed in \cite{Mypaper} that an effective method for calculating $\mu_f$ is not directly clear from the definitions given. However, calculation of $\omega_f$ should be much more efficient complexity-wise, since only a finite amount of values need to be checked to determine the minumum value. This quantity of values to check by brute force grows with complexity $O(q^n)$ and is therefore polynomial in $q$ (though exponential in $n$). Therefore, there is much value in the use of $\omega_f$ even when it is equal to $\mu_f \cdot (q-1)$.

\section{Future Work}

It is important to consider whether the results presented in Chapter \ref{refine} apply in more general settings. For instance, there are cases when it is more convenient to use rational interpolated form of a map than its polynomial form, especially when the monomials of the rational interpolation have much smaller degree. Even if we strictly considered Laurent polynomials, where we have $f(x) \in \mathbb{F}_q[x,x^{-1}]$ or the Laurent polynomial map $f(x_1,...,x_n) = (f_1(x_1,...,x_n),...,f_n(x_1,...,x_n))$ with $f_i(x_1,...,x_n) \in \mathbb{F}_q[x_1, x_2,...,x_n, x_1^{-1}, x_2^{-1},...,x_n^{-1}]$, can we apply the geometry of the Newton polytope to bound their cardinalities? Would such bounds be any stronger than those obtained by using a polynomial-interpolated form of the map?\footnote{Thanks to Matt Keti for compiling the following table.}

\begin{table}[h]
  \centering
	\caption{Examples of polynomials and their rational interpolations over $\mathbb{F}_{2^8}^*$. We denote $\alpha$ to be a multiplicative generator of $\mathbb{F}_{2^8}^*$.}
  \label{Interpolate}
\begin{equation*}
\begin{array}{c|c}
\text{Polynomial Interpolation} & \text{Rational Interpolation} \\ \hline \hline
x^{18}+3x^2+1 & \text{N/A} \\ \hline
x^{254} + x^{17} + 1 & (x^{18} + x + 1)/{x} \\ \hline
x^{254} + x^{253} + x^{30} & (x^{32} + x + 1)/{x^2} \\ \hline
(\alpha^6+\alpha^3+1)x^{254}+\ldots+(\alpha^6+\alpha^5) & x^{32}/(x^2+\alpha x+\alpha^7) \\ \hline
(\alpha^6+\alpha^5)x^{254}+\ldots+(\alpha^7+\alpha^6+\alpha^2) & (x^{88}+1)/(x^2+x+\alpha^5)
\end{array}
\end{equation*}
\end{table}